\documentclass[11pt]{amsart}

\usepackage[T1]{fontenc}

\usepackage{latexsym,enumerate}

\usepackage{mathrsfs}
\usepackage{amsmath,amssymb,amsthm,amsfonts,latexsym}
\usepackage[bookmarks]{hyperref}
\hypersetup{backref, colorlinks=true}
\hypersetup{backref, colorlinks=true, colorlinks   = true,
	urlcolor     = blue, linkcolor = blue, citecolor   = magenta
}
\def\adh#1{\overline{#1}}

\let\=\partial

\newtheorem {pro}{Proposition}[section]
\newtheorem {thm}[pro]{Theorem}%[section]
\newtheorem {cor}[pro]{Corollary}%[section]
\newtheorem{lem}[pro]{Lemma}

\theoremstyle{definition}
 \newtheorem {rem}[pro]{Remark}%[section]
\newtheorem {dfn}[pro]{Definition}%[section]

\newtheorem {obs}[pro]{Observation}

\newcommand{\sca}{\mathcal{S}}

\newcommand{\dri}{\mathring{d}}

\newcommand{\R}{\mathbb{R}}

\newcommand{\N}{\mathbb{N}}
\newcommand{\cc}{\mathscr{C}}

\newcommand{\et}{\quad \mbox{and} \quad }

\newcommand{\spl}{\mathcal{D}^+}

\newcommand{\bou}{ {\bf B}}

\newcommand{\id}{\mathring{d}}

\newcommand{\D}{\mathcal{D}}

\newcommand{\ep}{\varepsilon}

\newcommand{\pa}{\partial}

\newcommand{\om}{\infty}

\newcommand{\supp}{\mbox{\rm supp}}
\newcommand{\xo}{{x_0}}

\title[]{Inner Lipschitz approximation in o-minimal structures}

\makeatletter

\@namedef{subjclassname@2020}{%
  \textup{2020} Mathematics Subject Classification}

\makeatother

\author[N. Nguyen A. Valette and G. Valette]{Nhan Nguyen, Anna Valette and Guillaume Valette}

\address[N. Nguyen]{FPT University, Danang, Vietnam}
 \email{nguyenxuanvietnhan@gmail.com}
\address[A. Valette]{Katedra Teorii Optymalizacji i Sterowania, Wydzia\l\ Matematyki i Informatyki Uniwersytetu Jagiello\'nskiego, ul. S. \L ojasiewicza 6, 30-348 Krak\'ow, Poland}
 \email{anna.valette@im.uj.edu.pl}
\address[G. Valette]{Instytut Matematyki Uniwersytetu
Jagiello\'nskiego, ul. S. \L ojasiewicza 6, 30-348 Krak\'ow, Poland}
 \email{guillaume.valette@im.uj.edu.pl}


\thanks{The research cooperation was funded by the program Excellence Initiative -- Research University at the Jagiellonian University in Krak\'ow, the last two authors were supported by the Narodowe Centrum Nauki (Poland) under grant 2021/43/B/ST1/02359.}


\begin{document}

\begin{abstract}
Given an o-minimal structure,  we show that every definable (in this structure) mapping that is Lipschitz with respect to the inner metric can be approximated by
$\cc^1$ mappings that are  Lipschitz with respect to the inner metric with arbitrarily close bounds for the derivative. When the o-minimal structure
 admits $\cc^\om$ cell decomposition, we show that the approximation can be required to be  $\cc^\infty$ and we extend this result to outer Lipschitz mappings. The proof involves the construction of partitions of unity with sharp bounds for the derivative, which can be useful for other approximation problems.
\end{abstract}
\maketitle
\hfill{\it In the  memory of David Trotman}
\section{Introduction}
Approximation theorems play an important role in many areas of analysis, and in the recent years,  the development of o-minimal geometry provided us valuable tools to perform analysis on singular spaces.%  In  \cite{fischer}, A. Fischer  proved that Lipschitz definable mappings can be approximated by $\cc^1$ definable Lipschitz mappings with almost the same Lipschitz constant.

%It is well-known since the works of M. Shiota \cite{shiota} and J. Escribano \cite{escribano} that definable $\cc^1$ mappings can be approximated by smooth mappings with approximation of the derivative. Applications to analysis often require to work intrinsically and demand to have bounds for the first derivative.

In \cite{fischer}, A. Fisher shows that Lipschitz definable mappings can be approximated by continuously differentiable definable mappings with approximation of the Lipschitz constant. We show in this article an analogous result for inner Lipschitz definable mappings (Theorem \ref{B}). This is the class of mappings having bounded first derivative (almost everywhere), which appears naturally when performing analysis on singular varieties.   Definable approximations with bounded derivatives are for instance needed to define weight functions to study solutions of PDEs on singular domains (see for example \cite{kuplik}).
The proof of this theorem relies on the construction of definable $\cc^1$ partitions of unity with sharp bounds of the first derivative (Theorem \ref{A}), which is of its own interest and can be useful to show other density results. It can be seen as an improvement of a result of \cite{kcpv}, where $\Lambda_p^0$ regular partitions of unity were constructed. Somehow, we provide $\Lambda_1^0$-regular partitions of unity but with a constant which is arbitrarily small.

We also derive $\cc^\infty$ approximation results.
All the finiteness properties of o-minimal structures that make them very attractive also make them  very rigid.   
It is  well-known for instance that it is impossible to construct $\cc^\infty$ definable partitions of unity when the structure is polynomially bounded. In such structures, any $\cc^\infty$ function that has a $0$ Taylor series at some point must be $0$ on the entire connected component of its domain containing that point, which makes $\cc^\infty$ approximation problems much more challenging in this framework.
We strengthen Fisher's result by showing that it is indeed possible to approximate definable Lipschitz mappings by $\cc^\infty$ definable Lipschitz mappings (Theorem \ref{main}) in any o-minimal structure admitting $\cc^\infty$ cell decomposition.
 Moreover we give analogous results for inner Lipschitz mappings (Corollary \ref{cor_c_infty}).
% The proofs relies on a construction of a partition of unity (Theorem \ref{A} below) which is of its own interest.
 
We denote by $|.|$  the euclidean norm.  Given   a  set $A\subset\R^n$, we write $d(x,A)$ for the euclidean distance from a point $x\in\R^n$  to  $A$, and $\partial A$ for the topological boundary of  $A$, i.e. $\pa A:=\adh A\setminus A$.  $\bou(x,r)$ stands for the open ball of radius $r$ centered at $x $.
\section{Framework}
We recall a few basic facts and definitions about  o-minimal structures.
We refer to \cite{costeomin,dries_omin} for more.
In the sequel, although \cite{gv_livre} is only devoted to the subanalytic category, we will also sometimes refer to this survey for results whose proofs go over the framework of o-minimal structures.

 \subsection{O-minimal structures.}  \label{sect_ominimal}
  A \textit{structure} (expanding $(\mathbb{R}, + ,.)$) is a family $\mathcal{D} = (\D_n)_{n \in \mathbb{N}}$ such that for each $n$ the following properties hold
\begin{enumerate}
\item [(1)] $\D_n$ is a Boolean algebra of subsets of $\mathbb{R}^n$.
\item[(2)] If $A \in \D_n$ and $B\in \D_k$ then $ A\times B$  belongs to $\D_{n+k}$.
\item[(3)] $\D_n$ contains $\{x \in \mathbb{R}^n: P(x) = 0\}$, where $P \in \mathbb{R}[X_1,\ldots, X_n]$.
\item[(4)] If $A \in \D_n$ then $\pi(A)$ belongs to $\D_{n-1}$, where $\pi: \mathbb{R}^n \to \mathbb{R}^{n-1}$ is the standard projection onto the first $(n-1)$ coordinates.
 \end{enumerate}
 Such a family $\mathcal{D}$ is said to be \textit{o-minimal} if in addition:
 \begin{enumerate}
 \item [(5)] Any set $A\in \D_1$ is a finite union of intervals and points.
 \end{enumerate}
A set belonging to the structure $\mathcal{D}$ is called a {\it definable set} and a map whose graph is in the structure $\mathcal{D}$ is called  a {\it definable map}.

A structure $\mathcal{D}$ is said to be \textit{polynomially bounded} if for each definable function $f: \mathbb{R}\to \mathbb{R}$, there exist a positive number $a$ and  $p \in \mathbb{N}$ such that $|f(x)| < x^p$ for all $x > a$.

Examples of polynomially bounded o-minimal structures are the semi-algebraic sets and the globally subanalytic sets, but also the so-called $x^\lambda$-sets  as well as the structures defined by convenient quasi-analytic Denjoy-Carleman classes of functions \cite{rsw}.

We say that an o-minimal structure {\it admits $\cc^\om$ cell decomposition} if for each definable function $f  $ on an open  subset $U$ of $\R^n$, there is a definable open dense subset of $V$ of $U$ on which $f$ is $\cc^\om$.
All the just above mentioned examples of polynomially bounded o-minimal structures  admit  $\cc^\om$ cell decomposition. 

Throughout this article, $\mathcal{D} = (\D_n)_{n \in \mathbb{N}}$  will stand for a fixed polynomially bounded o-minimal structure (expanding $(\mathbb{R}, + ,.)$). The term ``definable'' will refer to this structure. 

Given  $A\in \D_n$ and $B\in \D_k$, we denote by $\D(A,B)$ the set of  definable  mappings $\xi:A\to B$ and by $\D(A)$ the set of  definable functions $\xi:A\to \R$, whereas  $\spl(A)$ will stand for the set of positive definable continuous functions on  $A$. 

Given $m\in \N\cup \{\infty\}$, we write $\D^m(A)$ (resp. $\D^m(A,B)$) for the elements of $\D(A)$  (resp. $\D(A,B)$) that are $\cc^m$ (a function will be said to be $\cc^m$ on  $A$ if it extends to a $\cc^m$ function on an open neighborhood of $A$ in $\R^n$).

%     \subsection{Definable Efroymson's  topology.}

%     and, given $\ep \in \spl(U)$,$$\U_\ep^1(f):=\{g\in \St^1(U,V):  |f-g|_1<\ep\}.$$
%     The {\it $\cc^1$ Efroymson topology} on $\D^1(U,V)$ is the topology
%     for which a basis of neighborhoods of $f$ is given by the family $(\U_\ep^1(f))_{\ep \in \spl(U)}$.

    Given a mapping $f:A\to B$, we denote by $|f|$ the function that assigns to $x\in A$ the number $|f(x)|$. When $f$ is differentiable,
      we will write $|df|$ for the function defined by $x\mapsto |d_xf|$, where $|d_x f|:=\sup_{|u|=1} |d_xf(u)|$. We then set $$|f|_1:=|f|+|df|.$$
      Recently \cite[Theorem $4.8$]{vv}, it was proved that if  $\D$ admits $\cc^\infty$ cell decomposition and if $M$ is a $\cc^\infty$ definable submanifold of $\R^n$ then
  for each  $f\in \D^1(
  M)$ and $\ep \in \D^+(M)$ there is $g\in \D^\om(M)$ such that on $M$
  $$|f-g|_1<\ep.$$
  This result was proved earlier in the semi-algebraic structure  by M. Shiota \cite{shiota}. The first deep insight in this direction was indeed achieved by G. Efroymson who showed that there is a $\cc^\infty$ semi-algebraic function $g$ such that $|f-g|<\ep$ \cite[Theorem $8.8.4$]{bcr}.

  %Efroymson's Approximation Theorem \cite[Theorem 8.8.4]{bcr} yields that for each $f\in \D^0(A)$ and $\ep \in \D^+(A)$ there is $g\in \D^\om(A)$ in the semi (for each semialgebraic set $A$).We will write $|d_xf|$ for the norm of $d_xf$ (as a linear mapping) derived from the Euclidean norm $|.|$.

%     Given  a definable mapping $\xi:A\to B$, with $A\in \D_n$, $B\in \D_k$, and $\ep \in \spl(A)$ we set: $$\U_\ep^0(\xi)=\{\zeta\in \St^0(A,B):  |\xi-\zeta|<\ep\}.$$
%     The {\it $\cc^0$ Efroymson topology} on $\D^0(A,B)$ is the topology
%     for which a basis of neighborhoods of $f$ is given by the family $(\U_\ep^0(f))_{\ep \in \spl(A)}$.

If $S$ is a definable $\cc^k$ submanifold of $\R^n$ ($k\ge 2$ possibly infinite) then there exist a definable open neighborhood $U_S$ of $S$ and a definable $\cc^{k-1}$
retraction $\pi_S:U_S\to S$ such that for all $x\in U_S$ we have $$d(x,S)=|x-\pi_S(x)|.$$
If $\dim S=n$ then  $U_S=S$ and $\pi_S(x)=x$ for all $x$.  Moreover, for $x\in U_S\setminus S$ we have $$\nabla_xd(\cdot,S)=\frac{x-\pi_S(x)}{|x-\pi_S(x)|},$$ which is always orthogonal to
$T_{\pi_S(x)}S $  (see \cite[Proposition 2.4.1]{gv_livre}).

Given  $\delta \in \D^+(S)$, we then let
 \begin{equation}\label{eq_udelta}
 	U_S^\delta:=\{x\in U_S:d(x,S)<\delta (\pi_S(x))\}.
 \end{equation}
%. The triplet $(U,\pi_A,d(\cdot ,A)^2)$ is called a tubular neighborhood of A

\subsection{Stratifications}\label{sect_stratifications}
We define the {\it angle} between two given vector subspaces $E$ and $F$ of $\R^n$  as:
$$\angle(E,F):=\underset{u\in E,|u|
	\le 1}{\sup} d(u,F).$$

\begin{dfn}\label{dfn_stratifications}
	A {\it $\cc^k$
		stratification of}\index{stratification} a definable set $A\subset \R^n$ is a finite partition of it into
	definable $\cc^k$ submanifolds of $\R^n$, called {\it strata}. A {\it refinement} of a stratification $\sca$ is a stratification $\sca'$ such that every $S\in \sca$ is a union of strata of $\sca'$.
	
	%If $\sca$ is a stratification of a set $A$ we will say that $(A,\sca)$ is a stratified set (by $\sca$).
% 	A stratification is {\bf compatible} with a set  if this set is the union of some strata. \index{compatible!  stratification}
	%We then say that $(X,\Sigma)$ is a {\bf stratified set}\index{stratified set}.A {\bf refinement of a stratification}\index{refinement! of a stratification} $\Sigma$ of $X$ is a stratification of $X$ compatible with every stratum of $\Sigma$.
%	We say that $\Sigma$ {\bf satisfies the frontier condition} if for every $S\in \Sigma$ the set $fr(S)\cap A$ is the union of some elements of $\Sigma$. 
%	

	 Let $S$ and $Y$ be a couple of strata (of some stratification)   and let $z \in S\cap \adh{Y}$.
	We will say that $(Y,S)$ satisfies the {\it $(w)$ condition at $z$} (of Kuo-Verdier) if there exists a constant $C$ such that for $y\in Y$ and $x
	\in S$ in a neighborhood of $z$:
	\begin{equation}\label{eq_w}
		\angle (T_x S,T_y Y) \leq C
		|x-y|.\end{equation}

A stratification $\sca$ is said to be {\it $(w)$-regular} if every couple  of strata $(Y,S)$ of $\sca$ satisfies the $(w)$ condition at every point of $S\cap \adh{Y}$.
We say that a stratification $\sca$ of a set $A$ {\it satisfies the frontier condition} if the closure in $A$ of every  $S\in \sca$ is a union of strata of $\sca$. Equivalently, given $S$ and $Y$ in $\sca$, either $\adh{Y}\cap S= \emptyset$ or $S\subset \adh {Y}$.    When strata are connected (which can always be obtained up to a refinement) and $A$ is locally closed, the $(w)$ condition entails the  frontier condition (see \cite{ver} or \cite[Propositions $2.6.3$ and $2.6.17$]{gv_livre}).  \end{dfn}%In the latter case, if $Y\ne S$ we then write $S\prec Y$.Since strata are definable, the latter case forces $\dim S<\dim Y$.

\begin{dfn}
We say that a couple of manifolds $(Y,S)$  satisfies {\it Whitney's $(b)$ condition} at  $x\in S$,
or is {\it $(b)$-regular} at $x$, if  for all sequences  $x_i\in S$ and $y_i\in Y$ with limit $x$
such that $T_{y_i} Y$ tends to $\tau$ and the lines $\overline{x_i y_i}$ tend to $\lambda$, one has $\lambda\subset \tau$.
When every pair of  strata  of a stratification is $(b)$-regular (at each point)
then we say that the stratification is {\it $(b)$-regular}.
\end{dfn}

It was shown by Ta Le Loi \cite{loi} that every definable set admits a $(w)$-regular stratification  and that Kuo-Verdier $(w)$ condition implies Whitney's $(b)$ condition.
 
 \begin{rem}\label{rem_whitney}
  If $M$ is a definable submanifold of $\R^n$ and $\xo \in \adh M \setminus M$ then $(M,\{\xo\})$ is  $(b)$-regular. Indeed, if Whitney's $(b)$  condition  failed, it easily follows from curve selection lemma that it would fail along a definable arc $\gamma$ in $M$ tending to $\xo$. But since $\gamma$ is definable, the angle between  $(\gamma(t)-\xo)$ and $\gamma'(t)\in T_{\gamma(t)}M$ must tend to $0$, which is a contradiction.
 \end{rem}

 We now recall the notion of horizontally $\cc^1$ stratified mapping due to C. Murolo and D. Trotman \cite{claudio}.
 %\begin{dfn}\label{dfn_stratified_mapping}
  \begin{dfn}
 Let $f:A\to B$ be a definable mapping. We say that $f:(A,\sca)\to (B,\mathcal{T})$  is a {\it stratified mapping}\index{stratified mapping} if $\sca$ and $\mathcal{T}$ are stratifications of $A$ and $B$ respectively such that every stratum  $S\in \sca$ is mapped by $f$ into an element  $T$ of $\mathcal{T}$ and the restricted mapping $f_{|S}:S\to T$ is a $\cc^\infty$ submersion.
%\end{dfn}

 Given a definable mapping $f:A\to B$, it is possible to construct stratifications $\sca$ and $\mathcal{T}$ of $A$ and $B$ respectively, making $f$ into a stratified mapping. These stratifications may satisfy the $(w)$ condition, or more generally any condition for which we can stratify definable sets (see \cite[Proposition 2.6.10]{gv_livre}).  

 A stratified mapping  $f:(A,\sca)
\to (B,\mathcal{T})$  is said to be {\it horizontally
$\cc^1$}  \index{horizontally $\cc^1$} if for any sequence
$(x_l)_{l\in \N}$ in a stratum $Y$ of $\sca$ tending to some
point $x$ in a stratum $S\in \sca$ and for any  sequence $u_l \in
T_{x_l} Y$ tending to a vector $u$ in $T_x S$, we have
$$\lim d_{x_l} f_{|Y}(u_l)=d_x f_{|S} (u).$$
\end{dfn}

\begin{rem}Since we can choose a sequence $u_l$ that tends to $0$, horizontally $\cc^1$ mappings must have locally bounded derivative. Moreover, when the $(w)$ condition holds, the required continuity condition in  the above definition entails that  every stratum $S$ has for each positive number $\mu$ a definable neighborhood $U$ in $A$ on which
\begin{equation}\label{eq_hc1_df}
	|d_{\pi_{S}(x)} f_{|S} | \le |d_x f_{|Y}|+\mu,
\end{equation} 
where $Y$ is the stratum that contains $x$. Conversely,
 given a mapping that has locally bounded derivative, it is not difficult to see that  there is a stratification of $f$ with respect to which $f$ is horizontally $\cc^1$. This stratification is indeed provided by a $(w)$-regular stratification of the restriction of the canonical projection $A\times B\to B$ to the graph of $f$. The corresponding  stratifications of $A$ and $B$ can be required to satisfy the $(w)$ condition (see \cite[Proposition $2.6.13$]{gv_livre} for a construction of regular horizontally $\cc^1$ stratifications).
\end{rem}

\subsection{Rugose vector fields}\label{sect_rugose} Let $\sca$ be a $(w)$-regular stratification of a definable set $A$ and let  $S\in \sca$  as well as $\xo\in S$.  It was established by Verdier \cite{ver} that when the $(w)$ condition holds,  each germ of smooth tangent vector field  on $S$ near $\xo$, say $v:U\cap S\to TS$, has an extension $\tilde v:U\cap A\to \R^n$, tangent to the strata (i.e. $v(x)\in T_xY$ for all $x\in U\cap Y$, $Y\in\sca$), for which there exists a constant $C$ (depending on $x_0$)  such that for all $x\in A$ and $y \in S$ in the vicinity of $\xo$:
 \begin{equation}\label{eq_rugose}
 |\tilde v(x)-\tilde v(y)|\le C |x-y|.\end{equation}
    Vector fields satisfying such an inequality are called {\it rugose} \cite{ver}. %This is a Lipschitz type condition but $x$ is any point of $ A$ close to $\xo$ while $y$ has to be in $S$.
    %More generally a rugose vector field $w$ on a stratified set $(X,\Sigma)$ is a mapping $w:X\to \R^n$ such that $w(x)\in T_x S$, for every $x\in S\in \Sigma$ and such that (\ref{eq_rugose}) for every couple of strata $S\prec Y$ for $y\in S$ and $x\in Y$ close to $y$.
%  Verdier proved that the $(w)$ condition implies that every
% vector field on a stratum $S$ extends to a rugose stratified vector field in a neighborhood
% of $S$???. (more details are given below???)
 This extension property is actually equivalent to the $(w)$ condition, as  established by Brodersen and Trotman \cite{bt}.
    For more details on stratifications and rugose vector fields we refer to the David Trotman's survey \cite{t} and references therein.

 \section{Inner Lipschitz mappings}
   If $x$ and $y$ are two points of the same connected component of a definable set  $A\subset\R^n$, we define
 $$\id_A(x,y):=\inf\{l(\gamma)\,:\, \gamma\in\D^0([0,1], A), \gamma(0)=x, \gamma(1)=y\},$$ where $l(\gamma)$ denotes the length of an arc $\gamma$ (defined as $\sup \sum_{i=0} ^{k-1} |\gamma(t_i)-\gamma(t_{i+1})|$, where the supremum is taken over all the partitions $0=t_0<\dots<t_k=1$).  When $x$ and $y$ do not belong to the same connected component, we put $\id_A(x,y):=+\infty$. We call $\id_A(\cdot,\cdot)$, the {\it inner metric on $A$}.

 The definition of the inner metric restricts itself to definable arcs because we need a bit of regularity in order to evaluate the length of arcs (definable arcs are piecewise $\cc^1$), and, as well-known, connected definable arcs are definably path connected. The set $A$ being possibly singular, we are not sure that there is always a $\cc^1$ arc connecting points. Lemma \ref{lem_distance_loc} will establish that  $\cc^0$ arcs would give the same metric (see (\ref{eq_inner_nondef})).

\begin{dfn}We say that a mapping $f:A\to\R^k$ is {\it inner Lipschitz} if there exists a constant $C$ such that the following inequality holds for any $x,y\in A$
 $$|f(x)-f(y)|\le C\mathring{d}_A(x,y).$$
%  We will say in that case that $f$ is $C$-inner Lipschitz.
The smallest such $C$ is then called {\it the inner Lipschitz constant of $f$}.

Given a  mapping $f:A\to\R^k$ and $\xo\in A$, we set
%the {local inner Lipschitz constant of $f$ at $\xo$}  is defined as
%We say that $f$ is {\it locally inner Lipschitz} at $x_0\in A$  if this point admits a neighborhood on which $f$ induces an inner Lipschitz mapping.  We then set
 $$\mathring{L}_f(x_0):=\inf_{r>0}\{C\,:\, |f(x)-f(y)|\le C\mathring{d}_A(x,y),\, \mbox{\rm for all}\, x,y\in \bou(x_0,r)\cap A\}.$$
It directly follows from the definitions that $f$ is inner  Lipschitz if and only if $\mathring{L}_f$ is a bounded function on $A$, and that in this case the inner Lipschitz constant of $f$ is equal to the supremum of $\mathring{L}_f$. \end{dfn}

We stress the fact that the inner distance is not necessarily definable. The  lemma below however establishes that  $\mathring{L}_f$ is  a definable function.
In this lemma, we consider $d_x f$ although $f$ is not assumed to be differentiable. This is because definable mappings are differentiable on a definable dense subset of their domain.

\begin{lem}\label{lem_der_bornee}
 Let $A\subset\R^n$ be a definable set and $f\in\D^0(A,\R^k)$. For all $\xo\in A$, we have %is inner Lipschitz  if and only if this point  $|d_x f|$ is bounded (almost everywhere) on $A$.
% $$\inf_{r>0}\sup\{ |d_xf|\,:\, x\in reg(f) \cap B(x_0,r)\}<\infty.$$
% In this case
\begin{equation}\label{lilc}
\mathring{L}_f(x_0)=\inf_{r>0}\sup\{ |d_xf|\,:\, x\in \bou(x_0,r)\cap A\}.\end{equation}\end{lem}

\begin{proof}
As $|d_xf|$ cannot be greater than the inner Lipschitz constant, the right-hand-side cannot be greater than the left-hand-side. It thus suffices to show the reversed inequality.

Take $(w)$-regular stratifications  $\sca$ and $\mathcal{T}$ of $A$ and $\R^k$ respectively  that make of  $f$ a horizontally $\cc^1$ mapping, and
 assume that $|d_x f|\le L$ on $ \bou(x_0,r)\cap A$ for some $r$.

 Let us show that there exists $\tilde{r}\le r$ such that $f$ is $L$-Lipschitz with respect to the inner metric on  $\bou(x_0,\tilde{r})\cap A$.
	By (\ref{eq_hc1_df}), we see that
	\begin{equation}\label{eq_borne_dfs}
	\sup_{x \in \bou(x_0,r)\cap S} |d_x f_{|S}| \le L,
	\end{equation}
	for every stratum $S\in\sca$ meeting this set. As  $\mathring{d}_A(a,\xo)$ tends to $0$ as $|a-\xo|$ tends to $0$ (see for instance \cite{kp} or \cite[Proposition $3.2.5$]{gv_livre}, this also directly follows from the next lemma), there is
	$ \tilde{r}\le \frac{r}{4}$ such that for all $a$ in
	$\bou(x_0,\tilde{r})\cap A$
	\begin{equation}\label{eq_d_axo}
	\mathring{d}_A(a,\xo)< \frac{r}{4}.
	\end{equation}
	 By definition of the inner metric, given  $a,b\in A$ there is for each $\mu>0$  a $\cc^0$ definable arc $\gamma:[0,1]\to  A$ of length not greater than $\mathring{d}_A(a,b)+\mu$ connecting $a$ and $b$. For $\mu<\frac{r}{4}$ and $a$ as well as $b$ in $\bou(x_0,\tilde{r})\cap A$, we then have
	$$l(\gamma)\le \mathring{d}_A(a,b)+\mu\le\mathring{d}_A(a,\xo)+\mathring{d}_A(b,\xo)+\frac{r}{4} \overset{(\ref{eq_d_axo})}<\frac{3r}{4},$$ which means that such an arc $\gamma$ does not leave $\bou(\xo,r)\cap A$ (since $\gamma(0)\in \bou(\xo,\frac{r}{4})$).
 Let then $t_0=0< t_1<\dots<t_k=1$ be such that  $\gamma$ stays in the same stratum on $(t_i,t_{i+1})$ for all $i<k$.  By (\ref{eq_borne_dfs}) and the Mean Value Theorem,  we have for every $i<k$: $$|f(\gamma(t_i))-f(\gamma(t_{i+1}))|\le L l(\gamma_{|[t_i,t_{i+1}]}), $$
	which shows that $|f(a)-f(b)|$ is not bigger than $L l(\gamma)$. Hence, $|f(a)-f(b)|\le L\mathring{d}_A(a,b)$, as claimed.
\end{proof}

The inner distance was defined by taking the infimum among all the {\it definable} continuous arcs. The proof of Lemma \ref{lem_strate_inner} will however involve
 integration of vector fields, which may generate non definable arcs. We thus need to check that the infimum among all the continuous  arcs is not smaller, for which the following lemma will be needed (see (\ref{eq_dring_dtilde})).
\begin{lem}\label{lem_distance_loc}
 Let $A\subset \R^n$ be a definable set and $\mu>0$. Every $x_0\in A$ has a neighborhood $U$ in $A$ such that each $x\in U$ can be connected to $\xo$ by a $\cc^0$ definable arc  in $A$ of length not greater than $(1+\mu)|x-x_0|. $
%
%  Consequently, for $x\in U$
%  \begin{equation}\label{eq_loc_inner}
% \mathring{d}_U(x,x_0)  \le (1+\mu)|x-x_0|.                                                                                               \end{equation}
\end{lem}
\begin{proof}
 We will assume $\xo=0$ and will argue by induction on $n$, the case $n=1$ being obvious.  Note first that, if $A_1,\dots, A_k$ are definable sets for which the lemma holds, then the lemma holds for their union. In particular, since  $A$ can be covered by finitely many sets that are Lipschitz cells after a possible orthonormal change of coordinates that we may identify with the identity (see for instance \cite{kp} or \cite[Proposition $3.2.5$]{gv_livre}), it is enough to show the result for $A$ of type $E\cup \{0\}$, where $E$ is a Lipschitz cell  satisfying $0\in \adh E$.
 %pb x\in adh A?

It suffices to construct a definable  mapping $h:U \times [0,1] \to U$, with $U$ neighborhood of $0$ in $A$, such that
\begin{enumerate}
 \item $h(x,0)=0$  and  $h(x,1)=x$
 \item $h(x,t)$ is $\cc^1$ with respect to $t$ and  the angle between the vectors $\frac{\pa h}{\pa t}(x,t)\ne 0$ and $h(x,t)$ tends to zero as $x\to 0$ (uniformly in $t\in (0,1]$).
\end{enumerate}

  %Take a Whitney stratification $\sca$ of $\adh E$ of which $E$ is a union of strata. all the $\pi(S)\cup 0_{\R^{n-1}}$,  $S\in \sca$,
 We denote by $h$ the mapping obtained by applying induction hypothesis to $\pi(A)$,
   where $\pi:\R^n \to \R^{n-1}$ is the projection omitting the last  coordinate.
We start with the case where $E$ is the graph of a $\cc^1$ Lipschitz function $\xi:\pi(E)\to \R$. In this case, we claim that
 $$\tilde h(x,t):=(h(x,t), \xi(h(x,t)))$$
has the required properties. To see this,
%take a sequence $(x_i, t_i)$ with  $x_i$ tending to zero and $t_i \in (0,1]$.
  let for simplicity $u(x,t)$ and $v(x,t)$ respectively denote the vectors $\tilde{h}(x,t)$ and $ \frac{\pa \tilde h}{\pa t}(x,t)$ divided by the norms of their respective projections onto $\R^{n-1}$.  Denote by $w(x,t)$ the vector realizing  the distance from $u(x,t)$ to $T_{\tilde{h}(x,t)}E$, $t>0$. Since $(E,0)$ is Whitney $(b)$-regular (see Remark \ref{rem_whitney}), $(u-w)(x,t)$ tends to zero as $x\to 0$.

 By induction $\pi(u)-\pi(v)$ tends to zero (as $x\to 0$), from which we deduce that so does $\pi(w)-\pi(v)$.  As $\xi$ is Lipschitz, the restriction of $\pi$ to  $T_{\tilde{h}(x,t)}E$ is a bi-Lipschitz mapping onto its image, with constants bounded independently of $x,t$.  Hence, $w-v$ goes to zero as $x\to 0$, and therefore so does $u-v$, as needed.

 We now have to address the case where $E$ is a band $\{x=(x', x_n)\in \pi(E)\times \R: \xi_1(x') <x_n<\xi_2(x')\}$, with $\xi_1<\xi_2$ Lipschitz $\cc^1$ functions on $\pi(E)$. %We proceed in the same way on both graphs,
 %taking Whitney stratifications $\sca_1$ and $\sca_2$ of the graphs of $\xi_1$ and $\xi_2$.  We then , for $B$ element of a common refinement of the partitions $\{\pi(S):S\in \sca_1\cup \sca_2\}$
%   applying the induction hypothesis to all the $\pi(A)$. This provides a mapping
In this case, we lift the path $h(x,t)$  as follows: given an element $x\in E$ that can be written
 \begin{equation}\label{eq_s}s\xi_1(\pi(x))+(1-s)\xi_2(\pi(x)), \;\; s\in (0,1), \end{equation}
  we set
  $$\tilde h(x,t):=(h(\pi(x),t), s\xi_1(h(\pi(x),t))+(1-s)\xi_2(h(\pi(x),t))).$$
To check that it has the required property, set for simplicity for $x\in \pi(E)$ and $i=1,2$
$$u_i(x,t):=\frac{\tilde{h}(x,\xi_i(x),t)}{|h(x,t)|}\et v_i(x,t):=\frac{\frac{\pa \tilde h }{\pa t} (x,\xi_i(x),t)}{|\frac{\pa  h}{\pa t}(x,t)|} .$$
% (resp. $v_i(x,t)$), $i=1,2$, the vector tangent to the graph of $\xi_i$ at $(h(x,t),\xi_i(h(x,t)))$ that projects onto  (resp. $\frac{\frac{\pa  h}{\pa t}(h(x,t)}{|\frac{\pa  h}{\pa t}(h(x,t)|}$).
% $$\frac{1}{|h(x,t)|}(h(x,t), \xi_i(h(x,t))
% \et \frac{1}{|\frac{\pa  h}{\pa t}(h(x,t))|} (\frac{\pa  h}{\pa t}(h(x,t)), d_{h(x,t)}\xi_i (\frac{\pa  h}{\pa t}(x,t)) ) .$$
The vectors   $u_i$ and $v_i$ project onto unit vectors and, by the above, the angle between these two vectors tends to $0$. Since the $\xi_i$'s are Lipschitz and the $v_i$'s are tangent to the graphs, this implies that $(u_i-v_i)$ tends to zero, $i=1,2$, which means that the difference between
$$w(x,t,s):=su_1(x,t)+(1-s)u_2(x,t) \et w'(x,t,s):=sv_1(x,t)+(1-s)v_2(x,t)$$ tends to zero as $x\to 0$ uniformly in $s\in (0,1)$ and $t$. The vector $w(\pi(x),t,s)$ (resp. $w'(\pi(x),t,s)$) being colinear to $\tilde h(x,t)$ (resp. to  $\frac{\pa \tilde h}{\pa t}(x,t)$) if $x$ and $s$ are as in (\ref{eq_s}), this yields the needed fact.
\end{proof}

Thanks to this lemma, we can show that the inner metric could be defined with non necessarily definable arcs, i.e., that given a definable set $A$ as well as $x$ and $y$ in $A$
\begin{equation}\label{eq_inner_nondef}
 \mathring{d}_A(x,y)=\inf \{l(\gamma): \gamma\in \cc^0([0,1], A) \mbox{  joining $x$ and $y$}\}.
\end{equation}
 If we denote by $\tilde{d}_A(x,y)$ the infimum that sits on the right-hand-side of the above equality then of course $\tilde{d}_A(x,y)\le\mathring{d}_A(x,y)$. To prove the reversed inequality, we will show that                                                                                                                                                                                                                                                                                                                                                                                                                                                                                                                            for each $\mu>0$
\begin{equation}\label{eq_dring_dtilde}\mathring{d}_A(x,y)\le (1+\mu)\tilde{d}_A(x,y).\end{equation}
 Take for this purpose an arc $\gamma\in \cc^0([0,1], A) $ joining two points $x$ and $y$ as well as $\mu>0$. Apply the just above lemma to the point $\gamma(t)$ for each $t \in [0,1]$ to get a neighborhood $U_t$ of $\gamma(t)$.
There is a finite partition $t_0=0<\dots<t_k=1$ such that $\gamma(t_i)$ and $\gamma(t_{i+1})$ are both in $U_{t_i}$ or in $U_{t_{i+1}}$, and thus can be joint by a $\cc^0$ definable arc  $\alpha_i:[t_i, t_{i+1}] \to A$ satisfying
\begin{equation}\label{eq_alpha_i}l(\alpha_i) \le (1+\mu)|\gamma(t_i)-\gamma(t_{i+1})|.\end{equation}
The $\alpha_i$ glue together into a continuous definable path $\alpha$ joining $x$ and $y$ and we have
$$\mathring{d}_A(x,y)\le  l(\alpha)=\sum_{i=0} ^{k-1} l(\alpha_i) \overset{(\ref{eq_alpha_i})}\le  (1+\mu)\sum_{i=0} ^{k-1}|\gamma(t_i)-\gamma(t_{i+1})| \le  (1+\mu)l(\gamma) ,$$
yielding (\ref{eq_dring_dtilde}), which establishes in turn (\ref{eq_inner_nondef}).

\begin{lem}\label{lem_angles}
For all vector subspaces $E,F$, and $T$ of $\R^n$ satisfying $ T \cap F^\perp=\{0\}$ we have
\begin{equation}\label{eq_angle_cap}
\angle (E\cap T^\perp, F\cap T^\perp)\le 2\;\frac{\angle(E,F)}{\inf_{v\in T, |v|=1}d(v,F^\perp) } .
\end{equation}

\end{lem}
\begin{proof}We denote by $\pi_V$ the orthogonal projection onto a space $V$. Take
a unit vector $u$ in $E\cap T^\perp$ and let us show that $d(u, F\cap T^\perp)$ is not greater than the right-hand-side of (\ref{eq_angle_cap}).

If $\pi_F(u)=\pi_{F\cap T^\perp}(u)$ then $d(u, F\cap T^\perp)=d(u, F)\le \angle(E,F)$ and the desired fact is clear. Otherwise, as the vector $\pi_F(u)-\pi_{F\cap T^\perp}(u)$ is the projection of $\pi_F(u)$ onto the orthogonal complement of $F\cap T^\perp$ in $F$, which is nothing but $\pi_F(T)$, we have,
\begin{equation}\label{eq_ufperp}
|\pi_F(u)-\pi_{F\cap T^\perp}(u)|=\langle \frac{\pi_F(u)-\pi_{F\cap T^\perp}(u)}{|\pi_F(u)-\pi_{F\cap T^\perp}(u)|},\pi_F(u)\rangle ,\end{equation}
and there is a unit vector $v$ in $T$ such that
\begin{equation}\label{eq_uv}
\frac{\pi_F(v)}{|\pi_F(v)|}= \frac{\pi_F(u)-\pi_{F\cap T^\perp}(u)}{|\pi_F(u)-\pi_{F\cap T^\perp}(u)|}.\end{equation}

Write now
\begin{eqnarray}\label{eq_u_moins_piqu}
|u-\pi_{F\cap T^\perp}(u)| &\le& |u-\pi_{F}(u)| +|\pi_F(u)-\pi_{F\cap T^\perp}(u)|\nonumber\\
&\le & \angle(E,F)+ |\langle\frac{\pi_F(v)}{|\pi_F(v)|},\pi_{F}(u)\rangle|,
\end{eqnarray}
by the definition of $\angle(E,F)$ as well as (\ref{eq_ufperp}) and (\ref{eq_uv}).
Observe now that
 $$ |\langle\frac{\pi_F(v)}{|\pi_F(v)|},\pi_{F}(u)\rangle|=|\langle\frac{ v}{|\pi_F(v)|},\pi_{F}(u)\rangle|=\frac{ |\langle v,u-\pi_{F}(u)\rangle |}{|\pi_F(v)|} \le \frac{\angle(E,F)}{d(v,F^\perp)}, $$
 because $|u-\pi_{F}(u)|\le \angle(E,F)$ and $|\pi_F(v)|=d(v,F^\perp)$. Plugging the just above inequality into (\ref{eq_u_moins_piqu})  yields the desired fact.
% stands for a unit vector of $Q$ normal to $Q\cap v^\perp$.
\end{proof}

\begin{lem}\label{lem_strate_inner}
Let $\sca$ be a $(w)$-regular stratification of a locally closed definable set $A\subset\R^n$. 	For every $S\in\sca$ and every $\mu>0$  there is $\delta \in \D^+(S)$ such that for all $x\in U^\delta_S\cap A$ \begin{equation}\label{eq_lem_strate_inner}
	\dri_A(x,\pi_S(x))\le (1+\mu)d(x,S). \end{equation}

	\end{lem}
\begin{proof} Let $\mu>0$ and $S\in \sca$.
We first construct a function $\delta\in \D^+(S)$ and a rugose vector field	$\eta:U^{\delta}_S\cap A \setminus S\to \R^n$ tangent to the strata (i.e. $\eta(x)\in T_x Y$ for all $x\in Y \cap U^\delta_S $, for all $Y\in \sca$ different from $S$) such that for all $x\in U^{\delta}_S\cap A \setminus S$ the following properties hold:
\begin{enumerate}
 \item\label{item_skal}  $\;\;\langle\nabla_x d(\cdot,S),\eta(x)\rangle=1.$ \item\label{item_norme}  $\;\;|\eta(x)|\le 1+\mu  $.
 \item\label{item_dpi} $\;\;d_x\pi_S(\eta (x))=0$.
\end{enumerate}
Let $\sca_k$ denote the set of all the strata of dimension less than or equal to $k$.  Let $l$ be the minimal integer such that $U^\delta_S\setminus S$ meets an $l$-dimensional stratum for arbitrarily small $\delta$.  We will define $\eta$ on the elements of $\sca_k$, by induction on $k\ge l$.

For $k=l$, we set for $x\in  U_S^\delta\cap Y, \,Y\in \sca_l$ different from $S$,
$$\eta(x):=P_x(\nabla_x d(\cdot,S)),$$
where $P_x$ is the orthogonal projection onto $T_{\pi_S(x)}S^\perp\cap T_x Y$.
Recall that $\nabla_x d(\cdot,S)=\frac{x-\pi_S(x)}{|x-\pi_S(x)|}$ and  $\ker d_x\pi_S=T_{\pi_S(x)} S^\perp$.    Since the $(w)$ condition implies  Whitney's $(b)$ condition, the angle between $ x-\pi_S(x)$ and $T_x Y$ tends to zero as $x$ draws near $S$, which by (\ref{eq_angle_cap}) (applied with $E=\nabla_x d(\cdot,S)\in T_{\pi_S(x)}S^\perp$, $F=T_xY$, $T=T_{\pi_S(x)}S$), means that $|P_x(\nabla_x d(\cdot,S))-\nabla_x d(\cdot,S)|$ is a definable function that tends to zero as $x$ tends to $S$. By Definable Choice \cite[Theorem 3.1]{costeomin}, we thus see that there is $\delta\in \D^+(S)$ such that on $U^\delta_S\cap Y$
\begin{equation}\label{eq_etarho} |P_x(\nabla_x d(\cdot,S))-\nabla_x d(\cdot,S)| <\mu .\end{equation}
which yields (\ref{item_norme}). Condition (\ref{item_dpi}) holds by construction. If we divide $\eta$ by $\langle\nabla_x d(\cdot,S),\eta(x)\rangle$, we get (\ref{item_skal}). Of course,  (\ref{item_norme}) is then affected by this division but since
$$|\langle\nabla_x d(\cdot,S),\eta(x)\rangle|\overset{(\ref{eq_etarho})}\ge 1-\mu  ,$$
we see that the right-hand-side of  (\ref{item_norme}) is only multiplied by a constant close to $1$,  and hence can be made  arbitrarily close to $1$ by choosing $\mu$ small enough. This completes the case $k=l$.

Take $k>l$ and
let $\eta$ be the vector field given by the induction hypothesis. Denote by  $\tilde{\eta}:U^\delta_S \cap A\setminus S \to \R^n$  its extension to a rugose vector field tangent to the strata and put
$$ \eta'(x):=\frac{P_x(\tilde{\eta}(x))}{\langle\nabla_x d(\cdot,S),P_x(\tilde{\eta}(x))\rangle},$$with $P_x$ as above.
We first check that $\eta'$ is rugose at every  $\xo \in Y\in \sca_{k-1}$, which reduces to check that so is $P_x(\tilde{\eta}(x))$. Fix $ Y\in \sca_{k-1}$ and $\xo \in Y$. Note first that  since $S$ is smooth $\angle (T_xS^\perp, T_{x'}S^\perp)$ is a locally Lipschitz function on $S$. Because $\inf_{|v|=1, v\in T_y Y^\perp} d(v, T_{\pi_S(y)}S)$ tends to $1$ as $y$ tends to $S$ (due to the $(w)$ condition), by
(\ref{eq_angle_cap}) (applied with $E=T_{\pi_S(y)}S^\perp$, $F=T_{\pi_S(x)}S^\perp$, $T=T_y Y^\perp$) we derive near $S$
\begin{equation}\label{eq_rugys}
 \angle (T_y Y\cap T_{\pi_S(y)}S^\perp,T_y Y\cap T_{\pi_S(x)}S^\perp) \le 4\angle ( T_{\pi_S(y)}S^\perp, T_{\pi_S(x)}S^\perp) \le C |y-x|,
\end{equation}
for some constant $C$. Moreover, (\ref{eq_angle_cap}) also entails that
if $Y'$ is a $k$-dimensional stratum such that $\xo\in \adh {Y'}\setminus Y'$ then for every $y\in Y$ and $x\in Y'$ close to $\xo$
\begin{equation}\label{eq_rugyy'}
 \angle (T_y Y\cap T_{\pi_S(x)}S^\perp,T_x Y'\cap T_{\pi_S(x)}S^\perp)\le 4\angle (T_y Y,T_x Y') \overset{(\ref{eq_w})}\le C|y-x|,
\end{equation}
for some possibly bigger constant $C$. As a matter of fact, we can write, using the triangle inequality for angles,
\begin{eqnarray*}
|(P_x-Id)_{|T_y Y\cap T_{\pi_S(y)}S^\perp} |= \angle (T_y Y\cap T_{\pi_S(y)}S^\perp,T_x Y'\cap T_{\pi_S(x)}S^\perp)  \end{eqnarray*}
\begin{eqnarray*}
&\le & \angle (T_y Y\cap T_{\pi_S(y)}S^\perp,T_y Y\cap T_{\pi_S(x)}S^\perp)+ \angle (T_y Y\cap T_{\pi_S(x)}S^\perp,T_x Y'\cap T_{\pi_S(x)}S^\perp)\\
&\le& 2C|y-x| \qquad \mbox{(by (\ref{eq_rugys})) and (\ref{eq_rugyy'})).}
\end{eqnarray*}
We deduce, thanks to the rugosity of $\tilde{\eta}$ and the fact that $\tilde{\eta}(y)\in
T_y Y\cap T_{\pi_S(y)}S^\perp$
$$|\tilde{\eta}(y)-P_x(\tilde{\eta}(x))|\le  |P_x(\tilde\eta(y))-\tilde{\eta}(y)|+|P_x(\tilde\eta(y))-P_x(\tilde\eta(x))|\le C|y-x| ,$$
for some constant $C$, yielding the rugosity of $\eta'$.

This vector field thus has the required properties on a small neighborhood $W$ of the strata of dimension smaller than $k$ (up to a slight change of $\mu$, (\ref{item_norme}), that holds for $\eta$ by the induction hypothesis, continues to hold for $\eta'$ on a small neighborhood of each element of $\sca_{l-1}$ by continuity).
%In order to glue $\eta'_{|W}$ with another vector field that has the required properties outside a neighborhood of these strata, 
Observe that if $Y_1,\dots, Y_\nu$ are the strata of dimension $k$  met by $U^\delta _S$ then
   the construction that we made in the case $k=l$ applies to every  $Y_i$, giving a  vector field $\eta'':\cup_{i=1} ^\nu Y_i\cap U^\delta_S\setminus S  \to \R^n$ satisfying $(1-3)$.  It thus suffices to glue $\eta'_{|W}$ and $\eta''$  by means of a partition of unity which is constant near each element of $\sca_{k-1}$ to complete our induction step.

Let $\phi(x,t)$ denote the integral curve of $-\eta$ starting at $x$. By condition (\ref{item_dpi}), this curve cannot leave $\pi_S^{-1}(\pi_S(x))$. Thanks to  (\ref{item_skal}), $\phi(x,t)$ exists up to $t=d(x,S)$ (the rugosity of $\eta$ prevents $\phi$ from falling earlier into another stratum, thanks to Gr\"onwall's inequality), and tends to   $\pi_S(x)$ as $t$ tends to this value. Condition (\ref{item_norme})  ensures that the length of this curve does not exceed $(1+\mu)d(x,S)$.  We then derive (\ref{eq_lem_strate_inner}) from (\ref{eq_inner_nondef}).
\end{proof}

%%%%%%%%%%%%%%%%%%%%%%%%%%%%%%%%%%%%%%%%%%%%%%%%%%%%%%%%%%%%%%%

\section{Approximations and partitions of unity}
We will prove simultaneously the two main results of this article (Theorems \ref{A} and \ref{B}). Constructing partitions of unity will require bump functions, for which the following lemma will be needed.
\begin{lem}\label{psi}
 Given  $\mu>0$, we can find a definable $\cc^1$ function $\psi_\mu:[0,+\infty)\times(0,1)\to [0,1]$
such that
\begin{enumerate}
\item $\psi_\mu(y,z)=\begin{cases}1& \mbox{\rm for}\; 0\le y\le b(z), \quad\mbox{\rm where } b(z) \in (0,z),\\
0  &\mbox{\rm for}\; y\ge z, 
\end{cases}$
\item $\left|\frac{\partial \psi_\mu}{\partial y}(y,z)\right|\le \frac{\mu}{y}$, \mbox{\rm for}\;  $y>0$,
\item $\left|\frac{\partial \psi_\mu}{\partial z}(y,z)\right|\le 4+\frac{\mu}{z}$.
\end{enumerate}

\end{lem}
\begin{proof}
Let $\mu$ be a positive rational number (as it suffices to prove this lemma for arbitrarily small values of $\mu$, we may assume $\mu$ to be rational). For $z\in (0,1)$ and  $y\in (0,z]$, we put  %$a(z):=\frac{z}{(z+1)^{1/\mu}}$, and  for $y\in (0,z)$
$$\phi_\mu(y,z)=(1+z)\left(1-\left(\frac{y}{z}\right)^\mu\right).$$
 We first estimate the partial derivatives of $\phi_\mu$. For every $0<y\le z $, we clearly have
\begin{equation}\label{eq_dphi_mu_y}
\left|\frac{\pa \phi_\mu}{\pa y}(y,z)\right|= (1+z)\frac{\mu}{y } \left(\frac{y}{z}\right)^\mu\le \frac{2\mu}{y }
\end{equation}
as well as %Finally, for each $y$ we have
\begin{equation}\label{eq_dphi_mu_z}\left|
	\frac{\pa \phi_\mu}{\pa z}(y,z)\right|=\left|1-\left(\frac{ y}{z}\right)^\mu+\frac{(1+z)\mu y^\mu}{z^{\mu+1}}\right|\le 1+\frac{(1+z)\mu}{z}. \end{equation}
Notice that $\phi_\mu(z,z)=0$ and $\phi_\mu(b(z),z)=1$, where $b(z):=\frac{z^{\frac{1}{\mu}+1}}{(1+z)^{\frac{1}{\mu}}}$. As a matter of fact, the function
$$\psi_\mu (y,z):=\begin{cases}
	1 & \mbox{if $0\le y<b(z)$}\\
	1-(1-\phi_\mu(y,z)^2)^2  & \mbox{if $b(z)\le y\le z$}\\
		0& \mbox{if $y>z$}\\
\end{cases}$$
is $\cc^1$ on $[0,+\infty)\times (0,1)$. Since $\mu$ is rational, $\psi_\mu$ is definable. As, for $y\in (b(z),z)$,
 $$d_{(y,z)} \psi_\mu =4\phi_\mu(y,z) (1-\phi_\mu(y,z)^2)d_{(y,z)}\phi_\mu(y,z), $$
 (\ref{eq_dphi_mu_y}) and (\ref{eq_dphi_mu_z}) yield
 $$ \left|\frac{\partial \psi_\mu}{\partial y}(y,z)\right|\le \frac{8\mu}{y}\et \left|\frac{\partial \psi_\mu}{\partial z}(y,z)\right|\le 4+\frac{4(1+z)\mu}{z}\le 4+\frac{8\mu}{z}, $$
 which is satisfying.  
\end{proof}
 \begin{thm}\label{A}
Let  $A$ be a locally closed definable set and  $\sca$  a $\cc^2$ stratification of $A$. Given   $\mu>0$ and for each $S\in\sca$  a definable neighborhood $U_S$ of $S$,
 we can find a definable $\cc^1$ partition of unity $\varphi_S:A\to[0,1]$, $S\in \sca$, subordinate to the covering $( U_S)_{S\in \sca}$ and such that for all $x\in U_S\setminus S$ \begin{equation}\label{eq_derpsi_thm}
|\nabla\varphi_S(x)|\le \frac{\mu}{d(x,S)} .
 \end{equation} 
 \end{thm}
 \begin{thm}\label{B}
Let  $A\subset \R^n$ be a locally closed definable set and $f:A\to\R^k$  an inner Lipschitz mapping. Given $\ep\in\D^+(A)$ and $\mu>0$,
we can find a definable $\cc^1$   inner Lipschitz mapping $g:A\to\R^k$ such that
\begin{enumerate}[(i)]
\item $|f-g|<\ep$,
\item $|d_xg|\le  \mathring{L}_f(x)+\mu$.
\end{enumerate}
\end{thm}
\begin{proof}Fix $0<\mu<1$. We prove these two theorems simultaneously  by induction on  $d:=\dim A$. Both statements being obvious for $d=0$, take $d\ge 1$ and denote by $\bold{A}_{<d}$ and $\bold{B}_{<d}$ the respective induction hypotheses of Theorems \ref{A} and \ref{B}.

We first point out a useful  consequence of $\bold{B}_{<d}$:

\noindent{\bf Observation.} Let $S$ be a manifold of dimension smaller than $d$, and $\delta \in \D^+(S) $.
By $\bold{B}_{<d}$, there is a $\cc^1$ definable function $\tilde\delta$ on $S$ satisfying $|d_x{\tilde\delta}|\le 2$ for all $x\in S$ and such that
$$|\tilde\delta-d(\cdot,\pa U^\delta_S)|<\frac{d(\cdot,\pa U^\delta_S)}{2}.$$
Note that then $\hat{\delta}:=\frac{\tilde\delta}{2}<\frac{3d(\cdot,\pa U^\delta_S)}{4}\le \frac{3\delta}{4}$, and therefore $U_S^{\hat \delta}\subset U^{\delta} _S$, which means that every $U^\delta _{S}$ contains a neighborhood defined by a $\cc^1$ definable  function $\hat \delta$ satisfying $|d_x\hat \delta|\le 1$.

We first perform the induction step of Theorem \ref{A}. Let $\sca$ and $(U_S)_{S\in \sca}$  be as in this theorem. Since it suffices to prove the theorem for a refinement of $\sca$, we can suppose  $\sca$ to satisfy the frontier condition.  We will assume the $U_S$ to be sufficiently small for the closest point retraction $\pi_S:U_S\to S$ to be a $\cc^1$ mapping satisfying  $|d_x\pi_S|<2$.
It suffices to construct a partition of unity subordinate to a family $(U_S^{\delta_S})_{S\in \sca}$ for arbitrarily small functions $\delta_S\in \D^+(S)$. %In particular, we will always assume that the functions $\delta_S$ that we choose satisfy $\delta_S<\min(\mu, d(\cdot, \adh S\setminus S))$.

% We may also require $\sca$ to be compatible with the $U_S$. In such a case every $S\in \sca$ has a neighborhood included in $U_S$. In this situation,

 Before constructing  the desired partition of unity, we need to define some bump functions $\psi_S^\delta$, for $S\in \sca$  and $\delta:S\to \R$ small enough positive $\cc^1$ definable function.
Assume first $\dim S<n$ and set for such   $\delta $
$$\psi_S^\delta (x):=\psi_\mu\left (d(x,S),\delta(\pi_S(x))\right),\quad x\in U_S,$$
where $\psi_\mu$ is given by Lemma \ref{psi}.
The function $\psi_S^\delta $ is equal to $1$ on $U^{b\circ \delta}_S$, where $b$ is as in the latter lemma. It  is zero near   $\pa U^\delta_S\setminus \pa S$ for $\delta $ small enough, and  $\cc^1$ on $U^\delta_S$  (the function $d(\cdot,S)$ is not smooth at points of $S$, but  $\psi_S^\delta$ being constant near $S$, it is obviously smooth at these points). We thus can extend $\psi_S^\delta $ outside $U^\delta_S$ (by $0$) to a $\cc^1$ function on the complement of $\pa S$.

We now wish to prove some estimates of  $d_x\psi_S^\delta$ that will be needed to establish (\ref{eq_derpsi_thm}). A computation of derivative yields
\begin{equation}\label{eq_dpsi_s_delta}
 d_x\psi_S^\delta =\frac{\pa \psi_\mu}{\pa y}(d(x,S),\delta(\pi_S(x)))d_xd(\cdot,S)+\frac{\pa \psi_\mu}{\pa z}(d(x,S),\delta(\pi_S(x)))d_{\pi_S(x)}\delta\, d_x\pi_S. \end{equation}
 As $d(\cdot,S)$ is $1$-Lipschitz,
  $(2)$  of  Lemma \ref{psi} entails that the norm of the first term of the right-hand-side is not greater than $\frac{\mu}{d(x,S)}$.  As for the second one, $(3)$ of the latter lemma implies that for $x\in U_S^\delta$
  $$|\frac{\pa \psi_\mu}{\pa z}(d(x,S),\delta(\pi_S(x)))|\le 4+\frac{\mu}{\delta(\pi_S(x))},$$
which is not greater than $\frac{5\mu}{ \delta(\pi_S(x))}$ if  $\delta<\mu$.  Hence, whenever $\delta$ is sufficiently small and satisfies in addition $|d_x\delta|<1$, we get altogether (from (\ref{eq_dpsi_s_delta}), using $|d_x\pi_S|\le 2$)
 for $x\in \R^n\setminus \adh S$  \begin{equation}\label{eq_der_psi}  |d_x\psi_S^\delta|\le \frac{11\mu}{d(x,S)}.\end{equation}
           In the case $\dim S=n$, we set $\psi_S^\delta(x)\equiv 1$ on $ S$, and  $\psi_S^\delta(x)\equiv 0$ on  $\R^n\setminus \adh S$ (we recall that $U^\delta_S=S$ for all $\delta$ in this case). The above estimate obviously continues to hold.

 We now turn to define the desired partition of unity. For $l\le d$ we denote by  $\sca_{l}$ the collection of the strata of $\sca$ of dimension  $\le l$, and set $X_l:=\bigcup_{S\in \sca_l}S$.
  We will define   $\varphi_S$ inductively on $\dim S$. Namely,  given $l\le d$,
we will construct $\cc^1$ functions $\varphi_S :\R^n\setminus \pa A\to [0,1]$, $S\in \sca_l$, satisfying
\begin{equation}\label{eq_sum_varphi_S}
\sum_{S\in \sca_l}\varphi_S(x)\le 1 \mbox{  for all } x \et  \sum_{S\in \sca_ l}\varphi_S(x)=1 \mbox{ on a neighborhood  of $X_l$},
\end{equation}
with  $\supp\,\varphi_S\subset U_{S}$  for each $S\in\sca_l$,  as well as for all $x\in \R^n\setminus (\pa A\cup \adh S)$:
	\begin{equation}\label{eq_item_der}|d_{x}\varphi_S|\le \frac{\mu}{d(x,S)}.\end{equation}
The desired partition of unity is given by the case $l=d$.

If $S$ is a $0$-dimensional stratum, just  set $$\varphi_S  (x):=\psi_S ^\delta (x) ,$$
which for $\delta\in \D^+(S) $ small enough,  has all the required properties.  Let us thus now define the functions $\varphi_S$, $S\in \sca_{l}\setminus \sca_{l-1}$, assuming all the $\varphi_Y$, $Y\in \sca_{l-1}$, to have already been defined.

 For $S\in \sca_{l}\setminus \sca_{l-1}$ and  $\delta_S:S\to \R$ small enough positive $\cc^1$ definable function,
 set
$$\varphi_S:=\psi_S^{\delta_S}\cdot (1-\sum_{Y \in \sca_{l-1}} \varphi_Y).$$
Thanks to the frontier condition, we may assume the $\delta_S$ to be sufficiently small for the $U_S^{\delta_S}$, $S\in  \sca_l\setminus \sca_{l-1}$, to be disjoint from each other. As a matter of fact, for every $S\in \sca_l\setminus \sca_{l-1}$,  on  $U_S^{b\circ\delta_S}$ we have  $\psi_S^{\delta_S}=1$ and $\psi_{S'}^{\delta_{S'}}=0$ if  $S'\in \sca_l\setminus \sca_{l-1}$ is different from $S$. Hence,  on $U_S^{b\circ\delta_S}$ we get
$$\sum_{Y\in \sca_l} \varphi_Y = \sum_{Y\in \sca_{l-1}} \varphi_Y +  \varphi_S  =\sum_{Y\in \sca_{l-1}} \varphi_Y +(1-\sum_{Y \in \sca_{l-1}} \varphi_Y)=1.$$
Moreover, as by induction $\sum_{Y \in \sca_{l-1}} \varphi_Y =1$ in the vicinity of $ X_{l-1}$, we have $\varphi_S=0$ near this set for each  $S\in \sca_l\setminus \sca_{l-1}$, so that near $X_{l-1}$ we have
$$ \sum_{Y\in \sca_l} \varphi_Y =\sum_{Y\in \sca_{l-1}} \varphi_Y =1.$$ 
Furthermore, since $0\le \psi_S^{\delta_S}\le 1$, $\sum_{S\in\sca_{l}}\varphi_S$ cannot exceed $1$, yielding (\ref{eq_sum_varphi_S}).

To check (\ref{eq_item_der}), observe that if the $\delta_S$ are sufficiently small  to have $d(x,S)\le d(x,Y)$ on $U_S^{\delta_S}$ when $Y$ is a stratum comprised in the boundary of a stratum $S\in \sca_l\setminus \sca_{l-1}$, then  the induction assumption entails for  $x\in U_S^{\delta_S}\setminus \adh S$
$$|d_x \varphi_Y|\le  \frac{\mu}{d(x,Y)}\le\frac{\mu}{d(x,S)}.$$
Thanks to the above observation, we may require  $|d_x\delta_S|<1$,
% which entails
% $$|d_x\psi_S^{\delta_S}| \overset{(\ref{eq_der_psi} )}\le \frac{11\mu}{d(x,S)},$$ so that
so that (\ref{eq_der_psi}) entails, together with the just above  inequality and the definition of $\varphi_S$, for $x\in U_S^{\delta_S}\setminus \adh S$,  $S\in \sca_l\setminus \sca_{l-1}$,
 $$|d_x \varphi_S|\le \frac{\kappa \mu}{d(x,S)}.$$
 for some constant $\kappa$  independent of $\mu$, yielding (\ref{eq_item_der})  ($\varphi_S$ being supported in $U_S^{\delta_S}$, this inequality continues to hold outside this set).
%Assume first $l<d$. In the case  $l=d$, just set for $S\in \sca_d$
%  $$\varphi_S:=1-\sum_{ Y\in \sca_{d-1}} \varphi_Y,$$
% and the required properties hold thanks to the  induction hypothesis.
 This completes our induction  on $l\le d$, giving the desired partition of unity and completing the induction step of the proof of Theorem \ref{A}.

 We now perform the induction step of Theorem \ref{B}. Fix $\ep \in \D^+(A)$.  Let $\sca$ and $\sca'$ be $\cc^2$ $(w)$-regular stratifications of $A$ and $\R^k$ with respect to which $f$ is a horizontally $\cc^1$ stratified mapping. Choose for each $S\in \sca$ a sufficiently small neighborhood $U_S$ of $S$ for the closest point retraction $\pi_S:U_S\to S$ to be well-defined and $\cc^1$. Let for  $S\in \sca$
 $$g_S(x):=f( \pi_S(x)), \quad x\in U _S.$$
%   We also set $g_S:=f$ if $\dim S=n$. satisfying $\dim S<n$
 Note that, as $f$ is continuous, it follows from Definable Choice that if $\delta_S\in \D^+(S)$    is sufficiently small  then  $|f-g_S|<\ep  $  on $U^{\delta_S}_S$.    As a matter of fact, if  we set $g:=\sum_{S\in \sca} \varphi_S g_S$, where $(\varphi_S)_{S\in \sca}$ is a partition of unity subordinate to the covering $(U_S^{\delta_S})_{S\in \sca}$ as given by Theorem \ref{A},  then
 $|f-g|=|\sum_{S\in \sca}\varphi_S(f-g_S)|<\ep. $

It thus  only remains to check that $g$ satisfies $(ii)$ for  small enough functions $\delta_S$. For this purpose, fix $\mu>0$ and denote by $P_x$ the orthogonal projection onto $T_xS$, for $x\in S\in \sca$. Note that since for each $S\in \sca$, $d_x \pi_S=P_{x}$   for every $x\in S$, it follows from Definable Choice that for $\delta_S$ small enough we have  for $x\in U_S^{\delta_S}$
 \begin{equation}\label{eq_d_x_pi}
 	|d_x \pi_S -P_{\pi_S(x)}|<\mu.
 \end{equation}
% Fix $x\in X$ and denote by $X$ the stratum that contains $X$.
% For  $\delta_S$ small enough, (\ref{eq_hc1_df})  entails that for $x\i n U_S$
%It follows from the frontier condition that if the $\delta_S$ are small enough then $U_S^{\delta_S}$ only meets $S$ and  higher dimensional strata.
 Set for simplicity $L=\sup_A \mathring{L}_f$.
 Notice also that by the definitions of $g_S$ and  $\mathring{L}_f $ we have for each $S\in \sca$ (for $\delta_S$ small enough) for $x\in U_S^{\delta_S}$
 \begin{equation}\label{eq_fmoinsg}
 |f(x)-g_S(x)|\le  L\cdot \dri_A(x,\pi_S(x)) \overset{(\ref{eq_lem_strate_inner})}{\le}  L(1+\mu)\cdot d(x,S).
 \end{equation}
Write now
\begin{equation}\label{eq_derg}|d_x\sum_{S\in \sca} g_S\varphi_S  | \le|\sum_{S\in \sca} g_S(x)d_x\varphi_S  |+|\sum_{S\in \sca}\varphi_S  (x)d_x g_S |.\end{equation}
Since $\sum_{S\in \sca} d_x\varphi_S=0$ on $A$, the  first term of the right-hand-side can be estimated as
\begin{equation}\label{eq_gidpsi} |\sum_{S\in \sca} g_S(x)d_x\varphi_S  | =   |\sum_{S\in \sca} (g_S-f)(x)d_x\varphi_S
 | \overset{(\ref{eq_fmoinsg})}\le L(1+\mu) \sum_{S\in \sca} d(x,S) |d_x\varphi_S
 |\overset{(\ref{eq_derpsi_thm})}\le L (1+\mu) \kappa \mu, \end{equation}
%  and therefore, for sufficiently small functions $\delta_S$,
% \begin{equation}\label{eq_gidpsi}
%  |\sum_{S\in \sca} g_S(x)d_x\varphi_S  |\overset{(\ref{eq_derpsi_thm})}\le L \cdot(1+\mu) \kappa \mu  ,
% \end{equation}
where $\kappa$ stands for the number of strata of $\sca$.
To estimate the second one, take $S\in \sca$ and write for $x\in  U_S ^{\delta_S}$
  %$v\in \R^n$ then
  \begin{equation*}|d_x g_S |= |d_{\pi_S(x)} f\, d_x \pi_S  |\overset{(\ref{eq_d_x_pi})}\le |d_{\pi_S(x)} f_{|S} \,P_{\pi_S(x)}|+  |d_{\pi_S(x)} f_{|S}|\mu \le  |d_{\pi_S(x)} f_{|S}|(1+\mu).
\end{equation*}
Hence, for $x\in  U_S ^{\delta_S}$  (with $\delta_S$ small enough), denoting by $Y$ the stratum that contains $x$, we get after applying (\ref{eq_hc1_df}) and multiplying by $\varphi_S$
% $$|d_x g_S |\overset{}\le  (|d_{x} f_{|X}|+\mu)(1+\mu),$$
% and therefore
 \begin{equation}\label{eq_dxgi}
  |\varphi_S(x) d_x g_S |  \le  \varphi_S(x)(|d_{x} f_{|Y}|+\mu)(1+\mu) \le \varphi_S(x)(\mathring{L}_f(x)+\mu)(1+\mu) .
 \end{equation}
  If $x\notin  U_S ^{\delta_S}$  then $\varphi_S(x)=0$   and this inequality is still valid.
  Hence, plugging (\ref{eq_gidpsi})   and (\ref{eq_dxgi})  into (\ref{eq_derg}), we get
$$|d_xg|\le L (1+\mu)\kappa \mu +(\mathring{L}_f(x) + \mu)(1+\mu) \le \mathring{L}_f(x) + C\mu   ,$$
for some constant $C$ independent of $\mu$.
\end{proof}

Thanks to the theorem proved in \cite{vv}, we get in addition:

\begin{cor}\label{cor_c_infty}
If the o-minimal structure admits $\cc^\infty$ cell decomposition, then the approximation provided by Theorem \ref{B} can be required to be $\cc^\infty$.
\end{cor}
\begin{proof}  Let $\ep \in \D^+(A)$ and $\mu>0$, and denote by $g$ the mapping provided by Theorem \ref{B}.
  By \cite[Theorem 4.8]{vv},  there exists $\tilde{g}\in\D^\infty(A,\R^k)$ such that $|\tilde{g}-g|_1\le \min(\ep,\mu)$. Clearly, $|f-\tilde{g}|<2\ep$ and $$|d_x\tilde{g}|\le |d_xg| +\mu\le  \mathring{L}_f(x)+2\mu.$$
  %As $\ep\le\mu$, by definition of $|\cdot|_1$ and (\ref{eq_gtilda}), this implies that  
%  $$\sup_{x\in A}|d_xg|< \mathring{L}_f(x)+\mu+\sup_{x\in A} \ep<\mathring{L}_f(x)+2\mu.$$
\end{proof}

\subsection{Approximation of Lipschitz definable mappings}
Given a Lipschitz mapping $f:A\to \R^k$, we denote by $Lip(f)$ its Lipschitz constant, i.e. $$Lip(f)=\inf\{L\in \R: \forall\, x,x'\in A,\quad \; |f(x)-f(x')|\le L|x-x'|\}.$$
%We begin with two key facts on Lipschitz functions.
\begin{obs}\label{lip}
Let $A\subset\R^n$ be open. If $f:\adh{A}\to\R^k$ is an inner Lipschitz mapping, $\cc^1$ on $A$  and   $L$-Lipschitz on $\partial A$ then $f$ is $L'$- Lipschitz on $\adh A$, where $L'=\max (\sup_{x\in A} |d_x f|, L)$.
\end{obs}
\begin{proof}
Let $x,y\in \adh A$ and define $$E:= \{t\in[0,1]:(1-t)x+ty\notin A\},\;\; t_0:=\inf E,\quad \mbox{\rm  and}\quad t_1:= \sup E.$$
 If $E$ is empty then the desired fact follows from the Mean Value Theorem. Otherwise, the corresponding points $a:=(1-t_0)x+t_0y$ and $b:=(1-t_1)x+t_1y$ being in the boundary of  $A$, we have by assumption  $|f(a)-f(b)|\le L|a-b|$.   By the Mean Value Theorem, we have $|f(x)-f(a)|\le K|x-a|$, where $K$ is the bound of the partial derivatives on $A$, and the same inequality holds for $t_1$ and the points $b$ and $y$. Finally, by the triangle inequality, we get $$|f(x)-f(y)|\le |f(x)-f(a)|+|f(a)-f(b)|+|f(b)-f(y)|\le K|x-a|+L|a-b|+K|b-y|,$$ which is not greater than  $$L'|x-a|+L'|a-b|+L'|b-y|=L'|x-y|,$$
where $L'=\max(L,K)$. 
\end{proof}

We first give a result in the  case where the domain of the mapping is a definable manifold.
\begin{thm}\label{main}Assume that the o-minimal structure admits $\cc^\infty$ cell decomposition.
Let $M\subset \R^n$ and $N\subset \R^k$ be  definable $\cc^\infty$ submanifolds. For any definable  Lipschitz mapping $f:M\to N$  and  any $\ep\in \spl(M)$ and $\mu>0$, there exists a  $(Lip(f)+\mu)$-Lipschitz mapping $g\in\mathcal{D}^\infty (M,N) $ such that
\begin{enumerate}[(i)]
 \item\label{item_i} $|f-g|<\ep$
 \item\label{item_ii} $|d_x g|\le  \mathring{L}_f(x)+\mu.$
\end{enumerate}
\end{thm}
\begin{proof}
 We first show that for each  $\mu<1$, the mapping $f$ can be extended to a  definable $(L+\mu)$-Lipschitz mapping $\tilde{f}$ on a neighborhood of $M$, which will reduce the proof to the case where $M$ is open in $\R^n$. Indeed, it suffices to show that $M$ and $N$ have neighborhoods $V$ and $W$ such that the closest point retractions $\pi_M:V\to M$ and  $\pi_N:W\to N$ are well-defined and $(1+\mu)$-Lipschitz, since the desired extension will then be given by $\tilde{f}:= \pi_N\circ f\circ \pi_M$.  We will focus on $M$, the argument for $N$ being analogous. Note first that since  $d_x \pi_M$ is for every $x\in M$ the orthogonal projection onto $T_xM$, we have $|d_x \pi_M|= 1$ for such $x$. By Definable Choice, there thus must be  $\delta \in \D^+(M)$ such that for each $x\in M$ the mapping $\pi_M$ is $(1+\mu)$-Lipschitz on $\bou(x,\delta(x))$. We claim  that $V:=U^{\mu\delta/4}_M$ is the desired neighborhood. Indeed, if $x$ and $y$ are two points of $V$ such that $|x-y|<\frac{\delta\circ\pi_M(x)}{2}$, they both belong to $\bou(\pi_M(x),\frac{1+\mu}{2}\delta\circ\pi_M(x))$, which for $\mu<1$ is included in $ \bou(\pi_M(x),\delta\circ\pi_M(x))$, on which $\pi_M$ is $(1+\mu)$-Lipschitz. Of course, the same argument applies when
$|x-y|<\frac{\delta\circ\pi_M (y)}{2}$. Finally, to address the case where  $|x-y|\ge \max(\frac{\delta\circ\pi_M(x)}{2},\frac{\delta\circ\pi_M(y)}{2})$, notice that for all $x$ and $y$ in $V$ we have
$$|\pi_M(x)-\pi_M(y)|\le |x-y|+|x-\pi_M(x)|+|y-\pi_M(y)|\le |x-y|+\frac{\mu}{4}\delta\circ\pi_M(x)+\frac{\mu}{4}\delta\circ\pi_M(y), $$
 which is not greater than $(1+\mu)|x-y|$ as soon as   $|x-y|\ge \max(\frac{\delta\circ\pi_M(x)}{2},\frac{\delta\circ\pi_M(y)}{2})$, showing that $\pi_M$ is $(1+\mu)$-Lipschitz on $V$.
 %Similarly, $\tilde{\ep}:=\ep\circ \pi_M$  is a continuous definable extension of $\ep$ to some  definable open neighborhood of  $M$.
Since it now suffices to prove the theorem for $\tilde{f}:=\pi_N\circ f\circ \pi_M$ with $\tilde{\ep}:=\ep\circ \pi_M$, we will assume from now that  $M$ is open without changing notations.

Fix $\mu>0$. Possibly replacing $\ep(x)$ with the function $\frac{\ep(x)d(x,\partial M)}{\sup\{d(a, \partial M)\, :\, |a|=|x|\}}\le \ep$, we can also assume that  $\ep(x)$ goes to $0$ as $x$ approaches $\partial M$, and possibly replacing it with $\min (\ep,\mu )$, we can assume it to be smaller than $\mu$. By Corollary \ref{cor_c_infty}, %\cite[Theorem 3]{fischer}
 there exists a definable $\cc^\infty$ inner Lipschitz mapping $g:M\to\R^k$ such that $|f-g|<\ep$ and
\begin{equation}\label{eq_gtilda}
	|d_xg|\le  \mathring{L}_f(x)+\mu\le  Lip(f)+\mu .\end{equation}
  Since $f$ is Lipschitz, it continuously extends to a Lipschitz mapping on $\adh M$ with the same Lipschitz constant. Since $\ep$ tends to zero at boundary points, $g(x):=f(x)$ extends continuously $g$ at boundary points, which means that  $g$ extends continuously to $\adh M$, inducing a $Lip(f)$-Lipschitz mapping on $\partial M$. By Observation \ref{lip} and (\ref{eq_gtilda}),  $g$ must be $(Lip(f)+\mu)$-Lipschitz on $M$.
\end{proof}

  Corollary \ref{cor_sing} will address the case where the mapping is defined on a possibly nonsmooth definable subset $A$ of $\R^n$. Its proof will make use of  the  definable version of Kirszbraun's theorem, which is a very delicate result established in \cite{AF}:
\begin{thm}\label{kirsz}
Let $A\subset\R^n$ be a definable set and $f:A\to\R^k$ a definable $L$-Lipschitz map.
There exists a definable  $L$-Lipschitz map  $\tilde f:\R^n\to \R^k$ such that $\tilde f_{|A}=f$.
\end{thm}
\begin{rem}\label{lext}
In the case $k=1$,  such an extension can however easily be obtained by setting $\tilde f(x):=\inf\{f(a)+L|x-a|\,:\, a\in A\}$. \end{rem}

\begin{cor}\label{cor_sing}
 Assume that the o-minimal structure admits $\cc^\infty$ cell decomposition.
Let $A\subset\R^n$ be a definable set. For any definable  Lipschitz mapping $f:A\to\R^k$  and  any $\ep\in \spl(A)$ and $\mu>0$, there exists a  $(Lip(f)+\mu)$-Lipschitz mapping $g\in\mathcal{D}^\infty (A,\R^k) $ such that $|f-g|<\ep$ on $A$.
\end{cor}
\begin{proof}
Thanks to Theorem \ref{kirsz}, we can extend $f$ to $\R^n$. Moreover,  we can extend $\ep$ continuously to a neighborhood of $A$  using  Riesz's formula (we may assume $\ep\le 1$)
\begin{equation*}%\label{eq_riesz}
 \tilde\ep(x)=\begin{cases}\sup\limits_{a\in A}\frac{\ep(a)d(x,A)}{|a-x|} & \mbox{\rm for}\; x\notin A\\
\ep(x) & \mbox{\rm for}\; x\in A\end{cases}
\end{equation*} The needed approximation is then provided by Theorem \ref{main}.
\end{proof}

\begin{rem}
	 In the case where $f$ is bounded, we can require that $Lip(g)=Lip(f)$. See \cite[Proof of Corollary $4$]{fischer} for details.
%		\item If $Lip_f(x)$ denotes the {\it local} Lipsc\begin{enumerate}
%	\itemhitz constant of $f:M\to \R^k$ at $x$ then, in the case where $A$ is open,  we can require in Theorem \ref{main}, in addition to (\ref{item_i}) and (\ref{item_ii}), that  This follows from \cite[Theorem 2]{fischer} and from the proof of Theorem \ref{main}.

%	\end{enumerate}
\end{rem}

Another application of Theorem \ref{main} is the following variant of Mostowski's separation lemma \cite{mostowski}.

\begin{cor}\label{cor_tdn}Let $M$ be a definable $\cc^\infty$ submanifold of $\R^n$ and let  $A, B$ be disjoint definable subsets of $M$. If $A$ and $B$ are closed in $M$ then there exists a Lipschitz function  $g\in \D^\om(M)$ which is positive on $A$ and negative on $B$.
 \end{cor}
\begin{proof}
 Let us define a Lipschitz function  by setting for $x\in M$, $f(x):=d(x,B)-d(x,A)$. By Theorem \ref{main}, there is a Lipschitz function $g\in \D^\om(M)$ such that $|f(x)-g(x)|<\frac{d(x,A)+d(x,B)}{2}$. This implies that $g(x)> \frac{d(x,B)}{2}$ on $A$, and $g(x)<- \frac{d(x,A)}{2}$ on $B$.
\end{proof}

\end{document}